\documentclass[a4paper,11pt]{amsart}

\usepackage{amssymb}
\usepackage{epsfig}
\usepackage{amsfonts}
\usepackage{amsmath}
\usepackage{euscript}
\usepackage{amscd}
\usepackage{amsthm}
\usepackage{bbm}
\usepackage[all,cmtip]{xy}
\DeclareMathAlphabet{\mathpzc}{OT1}{pzc}{m}{it}

\usepackage{soul}
\setstcolor{red}

\makeatletter
\@namedef{subjclassname@2020}{%
  \textup{2020} Mathematics Subject Classification}
\makeatother

\usepackage[colorinlistoftodos,textwidth=2cm,textsize=small]{todonotes}
\usepackage{verbatim}
\usepackage[pagebackref,hypertexnames=false, colorlinks, citecolor=red,linkcolor=blue, urlcolor=red]{hyperref}
\usepackage{mathrsfs}
\usepackage{tikz}

\theoremstyle{plain}

\newtheorem{thm}{Theorem}[section]
\newtheorem{cor}[thm]{Corollary}
\newtheorem{lem}[thm]{Lemma}
\newtheorem{prop}[thm]{Proposition}
\theoremstyle{definition}

\newtheorem{rem}[thm]{Remark}

\numberwithin{equation}{section}

\def\C{{\mathbb C}}
\def\D{{\mathbb D}}
\def\B{{\mathbb B}}

\def\H{{\mathcal H}}

\def\dbar{\bar\partial}

\def\inp#1,#2{\left\langle{#1},{#2}\right\rangle}

\def\N{\mathbb{N}}

\def\R{\mathbb{R}}

\def\si{\sigma}

\def\l{\lambda}

\def\L{\Lambda}
\def\ra{\rightarrow}
\def\ov{\overline}

\def\a{\alpha}
\def\b{\beta}
\def\g{\gamma}
\def\G{\Gamma}

\def\h{ hermitian holomorphic vector bundle}

\def\w{with respect to }

\def\beq{\begin{eqnarray}}
\def\eeq{\end{eqnarray}}
\def\beqa{\begin{eqnarray*}}
\def\eeqa{\end{eqnarray*}}

\def\del{\partial}

\def\M{\boldsymbol{M}}

\def\<{\langle}
\def\>{\rangle}

\def\bz{\boldsymbol{z}}
\def\bw{\boldsymbol{w}}

\newcommand{\be}{\begin{equation}}
\newcommand{\ee}{\end{equation}}
\newcommand{\bea}{\begin{eqnarray}}
\newcommand{\eea}{\end{eqnarray}}
\newcommand{\Bea}{\begin{eqnarray*}}
\newcommand{\Eea}{\end{eqnarray*}}

\newcounter{cnt1}
\newcounter{cnt2}
\newcounter{cnt3}
\newcommand{\blr}{\begin{list}{$($\roman{cnt1}$)$}
 {\usecounter{cnt1} \setlength{\topsep}{0pt}
 \setlength{\itemsep}{0pt}}}
\newcommand{\bla}{\begin{list}{$($\alph{cnt2}$)$}
 {\usecounter{cnt2} \setlength{\topsep}{0pt}
 \setlength{\itemsep}{0pt}}}
\newcommand{\bln}{\begin{list}{$($\arabic{cnt3}$)$}
 {\usecounter{cnt3} \setlength{\topsep}{0pt}
 \setlength{\itemsep}{0pt}}}
\newcommand{\el}{\end{list}}
\newcommand{\overbar}[1]{\mkern 1.5mu\overline{\mkern-1.5mu#1\mkern-1.5mu}\mkern 1.5mu}

\DeclareMathOperator{\Aut}{Aut}\DeclareMathOperator{\aut}{Aut}
\DeclareMathOperator{\mob}{\text{M\"ob}}

\newcounter{defcounter}
\begin{document}
\title[Homogeneous operators]{A Family of Homogeneous Operators In The Cowen-Douglas Class Over  The Poly-disc}
\author[P. Deb]{Prahllad Deb}
\address[P. Deb]{Department of Mathematics, Ben-Gurion University of the Negev, Beer-Sheva, 84105, Israel}
\email[P.Deb]{prahllad.deb@gmail.com} 
\author[S. Hazra]{Somnath Hazra}
\address[S. Hazra]{Mathematical Institute in Opava, Silesian University in Opava, Na Rybn\'icku 626/1, 74601, Opava, Czechia}
\email[S. Hazra]{somnath.hazra.2008@gmail.com} 

\thanks{The research of the first named author was supported by J C Bose National Fellowship of G. Misra (DSTO 1984) followed by a postdoctoral fellowship at Ben-Gurion University of the Negev, Israel (850053497). The research of the second named author was supported by a post-doctoral Fellowship of the NBHM followed by GA CR grant no. 21-27941S}

\begin{abstract} 
We construct a large family of positive-definite kernels $K: \mathbb{D}^n\times \mathbb{D}^n \to \mbox{\rm M} (r, \mathbb C)$, holomorphic in the first variable and anti-holomorphic in the second, that are quasi-invariant with respect to the subgroup $\mob \times\cdots\times \mob$ ($n$ times) of the bi-holomorphic automorphism group of $\mathbb{D}^n$.  The adjoint of the $n$ - tuples of multiplication operators by the co-ordinate functions on the Hilbert spaces $\mathcal H_K$ determined by $K$ is then homogeneous with respect to this subgroup. We show that these $n$ - tuples are irreducible, are in the Cowen-Douglas class $\mathrm B_r(\mathbb D^n)$ and that they are mutually pairwise unitarily inequivalent.
\end{abstract}


\maketitle


\section{Introduction}
Let $\mob$ denote the bi-holomorphic automorphism group of the unit disc $\D:=\{z\in \mathbb C: |z| < 1\}$. A bounded linear operator on a complex separable Hilbert space $\H$ is said to be homogeneous if the spectrum $\si(T)$ of $T$ is contained in $\overbar{\D}$ and, $g(T)$ is unitarily equivalent to $T$ for every $g$ in the M\"obius group $\mob$.  Indeed, since $g \in \mob$ is a rational function with pole outside the closed disc $\overbar{\mathbb D}$, it follows that $g(T)$ is well defined whenever the spectrum $\si(T)$ of $T$ is contained in $\overbar{\D}$. In a number of articles \cite{HOPRMGS, THS,  HOHS, ACHOCD, OICHO, HOJCS}, the class of homogeneous operators has been studied extensively. 

The notion of a homogeneous operator has a natural generalization to commuting tuples of operators. Let $\Omega$ be a bounded symmetric domain in $\C^n$ and $G$ be a subgroup of the bi-holomorphic automorphism group $\aut(\Omega)$ of $\Omega$. A commuting $n$ - tuple of operators $(T_1, T_2,\ldots ,T_n)$ is said to be homogeneous with respect to $G$ if the Taylor joint spectrum of $(T_1, T_2,\ldots , T_n)$ lies in $\overline{\Omega}$ and $g(T_1, T_2,\ldots , T_n)$, defined by the holomorphic functional calculus, is unitarily equivalent with $(T_1, T_2,\ldots ,T_n)$ for all $g \in G$. In \cite{HHVCDBSYMD}, Kor\'{a}nyi and Misra described all irreducible homogeneous tuples of operators in the Cowen-Douglas class over the open unit ball $\B^n := \{z = (z_1, \ldots, z_n) \in \C^n : \sum |z_i|^2 < 1\}$, see also \cite{HVBIOSD}. In this generality, the focus has been on the study of commuting tuples of homogeneous operators in the Cowen-Douglas class over an irreducible bounded symmetric domain. In a recent paper, we initiated the study of $n$ - tuples of operators homogeneous under the group  $\mob^n:=\mob \times \cdots \times \mob$.  The group  $\mob^n$ is evidently a subgroup of the bi-holomorphic automorphism group of $\mathbb D^n$ and acting transitively on $\mathbb D^n$.



In the recent paper \cite{HHHVB1}, all irreducible $n$ - tuples in $\mathrm B_r(\D^n)$, $r = 1, 2, 3$, homogeneous with respect to the group M\"ob$^n$ (respectively, $\Aut(\D^n)$) have been listed modulo unitary equivalence. It has been also noted, in the same article, that the representations associated to the $n$ - tuples of operators in $\mathrm B_r(\D^n)$ homogeneous with respect to M\"ob$^n$, $r =1, 2, 3$,  must be multiplicity-free. 

In general, the classification modulo unitary equivalence of all homogeneous tuples of operators in $\mathrm B_r(\D^n)$ for $r > 3$ is rather complicated.  Recall that if $T_i$, $i=1,\ldots, n-1$ are  homogeneous operators in $\mathrm B_1(\D)$ and $S$ is an irreducible homogeneous operator in $\mathrm B_r(\D)$, $r = 1$ or $2$, then setting $\hat{T}_i = I\otimes \cdots \otimes I \otimes T_i \otimes I \otimes \cdots \otimes I$ and $\hat{S}:= I\otimes \cdots \otimes I \otimes  S$, we infer that the operators of the form $(\hat{T}_1, \ldots , \hat{T}_{n-1}, \hat{S})$ are the only irreducible $n$ - tuples of operators (up to unitary equivalence) in $\mathrm B_r(\D^n)$, $r=1,2$, homogeneous \w $\mob^n$ (cf. \cite[Theorem 3.1, Theorem 6.7]{HHHVB1}).  

Continuing in the same manner, for $r \geq 3$, examples of irreducible $n$ - tuples $\boldsymbol{T} \in \mathrm B_r(\D^n)$ homogeneous under $\mob^n$ can be produced by taking tensor product, namely,
\begin{equation} \label{tensor}
\boldsymbol{T} = (\hat{T}^{(n_1)}, \ldots , \hat{T}^{(n_p)}),\,\,\mbox{\rm with}\,\, \hat{T}^{(n_i)} = I\otimes \cdots \otimes I \otimes \boldsymbol{T}^{(n_i)} \otimes I \otimes \cdots \otimes I,\,\, \boldsymbol{T}^{(n_i)}\in\mathrm B_{r_i} (\mathbb D^{n_i})\end{equation} 
such that $n_1 + \cdots + n_p=n$, $r_1 \cdots r_{p} = r$. Here, we use the notation 
$$ I \otimes \cdots \otimes I \otimes \boldsymbol{T}^{( n_i )} \otimes I \otimes \cdots \otimes I \!: = \!( I \otimes \cdots \otimes I \otimes T_{1 n_i}  \otimes I \otimes \cdots \otimes I, \! \cdots, \!  I \otimes \cdots \otimes I \otimes T_{n_i n_i}\otimes I \otimes \cdots \otimes I ) $$
where $ \boldsymbol{T}^{( n_i )} = ( T_{1 n_i}, \hdots, T_{ n_i n_i } ) \in \mathrm B_{r_i} (\mathbb D^{n_i}) $.
Now consider the class $ \mathcal{T} $ of $n$ - tuples of operators $ \boldsymbol{T} \in \mathrm B_r( \D^n ) $ obtained as in \eqref{tensor} and let $ \mathcal{T}_1 $ be the subset of $ \mathcal{T} $ consisting of all operator tuples $ \boldsymbol{T} =  (\hat{T}^{(n_1)}, \ldots , \hat{T}^{(n_p)})$ with $ n_i = 1 $, $ 1 \leq i \leq p $. Then unlike the case when $r=1$ or $2$, it was shown in \cite[Theorem 7.8]{HHHVB1} that there are $n$ - tuples 
in $\mathrm B_3(\D^n)$ homogeneous under $\mob^n$ are not necessarily in $ \mathcal{T}_1 $.  The specific example for $r=3$ given in that paper, however, is in the larger set $ \mathcal{T} $. Never the less, we don't know if $ \mathcal{T} $ consists of all irreducible $n$ - tuples of operators in $\mathrm B_r(\D^n)$ homogeneous under $\mob^n$. 

It is useful to also give an alternative description of some of these operators as $n$ - tuples of multiplication of operators on a reproducing kernel Hilbert space. This is the class obtained by setting $n_i=1$, $i=1, \ldots ,n$ and $r_1 = \cdots = r_{n-1}=1$ and $r_n=r$.  For $\l_1,\hdots,\l_n>0$, let $\mathcal{H}^{(\l_i)}(\D)$ be the reproducing kernel Hilbert space with the reproducing kernel $K^{(\l_i)} (z, w) = (1 - z\bar{w})^{-\l_i},\,\,z, w \in \mathbb{D}$ and $A^{(\l_n, \mu)}$ be the reproducing kernel Hilbert space determined by the positive definite kernel $B^{(\l_n, \mu)}$ as in \cite[Equation (4.3)]{HOHS}. Then the adjoint of the $n$ - tuple of multiplication operators by the co-ordinate functions on the reproducing kernel Hilbert space $\mathcal{H}^{(\l_1)} \otimes \mathcal{H}^{(\l_2)} \otimes \cdots \otimes \mathcal{H}^{(\l_{n-1})} \otimes A^{(\l_n, \mu)}$ is in $\mathrm B_r(\D^n)$, it is irreducible and  homogeneous with respect to M\"ob$^n$ and are mutually unitarily inequivalent.  

%

Here, we construct a family of irreducible $n$ - tuples of operators in $\mathrm B_r(\D^n)$  homogeneous \w $\mob^n$ that is not of the form \eqref{tensor} if all the $n_i$'s are required to be $1$. This, we hope, would be a first step in any attempt to classify homogeneous $n$ - tuples in $\mathrm B_r(\D^n)$, $r >3$.  The description of the Hilbert spaces in \cite[Section 7]{HHHVB1} prompts a natural generalisation of the $\G$ map leading to construction of $n$ - tuples of operators in $\mathrm B_r(\D^n)$, $r\geq 3$, in a manner very similar to the one in  \cite{HOHS}. 


In Section \ref{construction}, for a fixed but arbitrary $r$ - tuples of positive real numbers $\mu:=(\mu_1,\hdots,\mu_r)$, we introduce ``the $\G$ map" depending 
on $\mu$, from $\oplus_{0\leq\b\leq\a}(\otimes_{i=1}^n\H^{(\l_i+2\b_i)}(\D))$ taking values in Hol$(\D^n,\C^r)$. Here $\l=(\l_1,\hdots,\l_n)$ is a fixed but arbitrary $n$ - tuple of positive real numbers, $\a=(\a_1,\a_2,\hdots,\a_n)\in(\N\cup\{0\})^n$ such that the cardinality of the set $\L=\{\b\in(\N\cup\{0\})^n:\b\leq \a\}$ is $r$, where  $``\leq"$ denotes the graded co-lexicographic ordering on $(\N\cup\{0\})^n$. Choosing the inner product in the image of $\G$ that makes $\G$ unitary,  we see that the image of $\G$  is a reproducing kernel Hilbert space $\H^{(\l,\mu)}$. The reproducing kernel of the Hilbert space $\mathcal{H}^{(\l,\mu)}$ is computed explicitly. We then prove that the $n$ - tuple of multiplication operators $\M^{(\l,\mu)}:=(M_{z_1}, M_{z_2}, \ldots, M_{z_n})$ by the co-ordinate functions on $\mathcal{H}^{(\l,\mu)}$ is M\"ob$^n$ - homogeneous by obtaining the associated multiplier representation of $\mob^n$ explicitly. We also establish that the adjoint of the $n$ - tuple of multiplication operators $\M^{(\l,\mu)}$ is in $\mathrm B_r(\D^n)$. In Section \ref{irreducible}, we show that the tuple of multiplication operators $\M^{(\l,\mu)}$ on $\mathcal{H}^{(\l,\mu)}$ is irreducible. In the final section, we prove that $\M^{(\l,\mu)}$ and $\M^{(\l',\mu')}$ are unitarily equivalent if and only if $\l=\l'$ and $\mu=\mu'$. 

\section{Preliminaries}

Let $\D$ be the open unit disc in $\C$ and $\mob$ be the set of all bi-holomorphic automorphisms $z\mapsto e^{i\theta}\frac{z-a}{1-\ov{a}z}$, $a\in\D$, $\theta\in[0,2\pi)$ of $\D$.  Note that SU$(1,1):=\left\{\left(\begin{smallmatrix}a&b\\\ov{b}&\ov{a}\end{smallmatrix}\right):|a|^2-|b|^2=1\right\}$ is a two fold cover of $\mob$. Consequently, the universal covering space $\widetilde{\mob}$ of $\mob$ is same as that of SU$(1,1)$. Let $\D^n:=\{(z_1, \ldots , z_n) : |z_1|, \ldots , |z_n| < 1\}$ and $\mob^n:=\mob\times\cdots\times\mob$. It is known that bi-holomorphic automorphisms $\aut(\D^n)$ of $\D^n$ is the semi-direct product $\mob^n \rtimes \mathfrak S_n$, where $\mathfrak S_n$ is the permutation group on $n$ symbols. As we have said before, the group $\mob^n$ acts transitively on $\mathbb D^n$.  In this paper, we study $n$ - tuples of multiplication operators homogeneous under the group $\mob^n$. 



Let $\H_K$ be a Hilbert space of holomorphic functions on $\D^n$ taking values in $\C^r$ for some $r\in\mathbb{N}$ possessing a reproducing kernel $K$ satisfying the reproducing property: for each $z\in\D^n$, $\xi\in\C^r$ and $f\in \H_K$, 
\beq \label{rep property}\<f,K(.,z)\xi\>_{\H_K}=\<f(z),\xi\>_{\C^r}.\eeq 
Note that $K(.,z)\xi\in\H_K$ for each $z\in\D^n$ and $\xi\in\C^r$ and as is well-known, the linear span of the set $\{K(.,z)\xi:z\in\D^n,\xi\in\C^r\}$ is dense in $\H_K$. Also, by differentiating both sides of the Equation \eqref{rep property}, we have that 
\beq \label{higher order rep property} \<f,\dbar^l_iK(.,z)\xi\>_{\H_K}=\<\del^l_i f(z),\xi\>_{\C^r},~f\in\H_K,~1\leq i\leq n,~l\in\N,~\text{and}~\xi\in\C^r.\eeq

Recall from \cite{OPOSE,GBKCD} that the adjoint $\M^*$ of the $n$ - tuple of multiplication operators $\M$ by the co-ordinate functions on $\H_K$ is in the Cowen-Douglas class $\mathrm B_r(\D^n)$ over $\D^n$ if the operator $D_{\M^*} : \H_K \to \H_K \oplus \cdots \oplus \H_K$ defined by $D_{\M^*} f = \left(M_{z_1}^* f,\hdots,M_{z_n}^* f\right)$, $f \in \H_K$ satisfies the following properties:
\begin{itemize} 
\item $\dim \ker D_{\M^* - z I}=r$, $z\in \D^n$;
\item  $\text{ran} D_{\M^* - z I}$ is closed in $\H_K\oplus\cdots\oplus\H_K$;
\item $\bigvee_{\bz\in \D^n} \ker D_{\M^* - z I}$ is dense in $\mathcal{H}_K$.
\end{itemize}
In this case, there is a {\h} $E_{\M^*}:=\{(z,v):v\in\ker D_{\M^*-zI}\}\subset\D^n\times\H_K$ defined over $\D^n$ such that the unitary equivalence class of the $n$ - tuple $\M^*$ and the (local) equivalence class of the hermitian holomorphic  vector bundle $E_{\M^*}$ determine each other. Given two reproducing kernels $K_1,K_2:\D^n\ra\text{M}(r,\C)$ with the property that $\M_i^*$ on $\H_{K_i}$, $i=1,2$ are in $\mathrm B_r(\D^n)$, it follows that $M_1^*$ and $M_2^*$ are unitarily equivalent if and only if there exists a holomorphic function $\Phi:\D^n\ra\text{GL}(r,\C)$ such that $K_1(z,w)=\Phi(z)K_2(z,w)\Phi(w)^*$ which is same as the fact that the {\h}s $E_{\M_1^*}$ and $E_{\M_2^*}$ are unitarily equivalent over $\D^n$. 

If in addition, the $n$ - tuple of multiplication operators $\M=(M_{z_1},\hdots,M_{z_n})$ by the co-ordinate functions on $\H_K$ is homogeneous \w $\mob^n$ -- that is, for each $\phi=(\phi_1,\hdots,\phi_n)\in\mob^n$, $\phi(\M)=(M_{\phi_1},\hdots,M_{\phi_n})$, where  $\phi(\M)$ is defined by the usual holomorphic functional calculus, is unitarily equivalent to $\M$ -- then  the {\h} $E_{\M^*}$ is homogeneous \w $\mob^n$ as well.  A hermitian holomorphic vector bundle $\pi:E_{\M^*}\ra\D^n$ is said to be \textit{homogeneous} \w a subgroup $G$ of the group of bi-holomorphic automorphisms of $\D^n$ if for every $\phi\in G$, there is an isometric bundle automorphism $\hat{\phi}:E_{\M^*}\ra E_{\M^*}$ such that $\phi\circ\pi=\pi\circ\hat{\phi}$. In this paper, the group $G$ is fixed once for all to be the subgroup $\mob^n$ of the group of bi-holomorphic automorphisms of $\D^n$. Moreover, it follows from Theorem 4.1 in \cite{HHHVB1} that the universal covering group $\widetilde{\mob}^n$ acts on $E_{\M^*}$ uniquely by isometric bundle automorphisms. 

Also, we recall from \cite[Section 4]{HHHVB1} that the homogeneity of $\M^*\in\mathrm B_r(\D^n)$ makes the reproducing kernel $K$ \emph {quasi-invariant}, that is,  \beq \label{quasi-invariant} K(z,w)=J(\tilde{\phi},z)K(\phi(z),\phi(w))J(\tilde{\phi},w)^*,~z,w\in\D^n,~\tilde{\phi}\in\widetilde{\mob}^n,\eeq where $\phi=p(\tilde{\phi})$, $p:\widetilde{\mob}^n\ra\mob^n$ is the universal covering map and for each $\tilde{\phi}\in \widetilde{\mob}^n$,  $J(\tilde{\phi},.):\D^n\ra\text{GL}(r,\C)$ is a holomorphic mapping. Further, if the function $J$ satisfies the ``co-cycle" identity $$J(\tilde{\phi}\tilde{\psi},z)=J(\tilde{\psi},z)J(\tilde{\phi},\psi(z)),~\tilde{\phi},\tilde{\psi}\in\widetilde{\mob}^n,~\text{and}~z\in\D^n,$$ we say that $J$ is a \textit{co-cycle}. 
If $K$ is quasi-invariant and  J is a co-cycle, then it is easy to verify that the map $U:\widetilde{\mob}^n\times\H_K\ra\H_K$ defined by \beq \label{mult-rep}(U(\tilde{\phi}^{-1})f)(z):=J(\tilde{\phi},z)f(\phi(z))\eeq is a unitary representation of $\widetilde{\mob}^n$ onto $\H_K$.

\section{Construction of Hilbert spaces}\label{construction}

In this section, we construct a family of reproducing kernel Hilbert spaces consisting of holomorphic functions on $\D^n$ taking values in $\C^r$ for a given $r\in \N$.  The reproducing kernels associated to these Hilbert spaces are computed explicitly. We then prove that the $n$ - tuples of multiplication operators by the co-ordinate functions on these Hilbert spaces are bounded. Finally, the adjoint of these $n$ - tuples of multiplication operators are shown to be in $\mathrm B_r(\D^n)$.


For a given $r \in \N$, fix an element $\a=(\a_1,\hdots,\a_n)\in(\N\cup\{0\})^n$ so that the cardinality $|\L|$ of the set $\L=\{\b\in(\N\cup\{0\})^n:\b\leq \a\}$ is $r$. Here we are using the graded co-lexicographic ordering on $(\N\cup\{0\})^n$. 


For each $\b=(\b_1,\hdots,\b_n)\in\L$ and $\l=(\l_1,\hdots,\l_n)$ with each $\l_i>0$, consider the reproducing kernel Hilbert space $\H^{(\l+2\b)}:=\bigotimes_{i=1}^n\H^{(\l_i+2\b_i)}(\D)$ of holomorphic functions on $\D^n$ with the reproducing kernel 
\beq \label{rep ker1} 
K^{(\l+2\b)}(z,w)=\prod_{i=1}^n(1-z_i\ov{w}_i)^{-(\l_i+2\b_i)},~z,w\in\D^n.\eeq 
Define the mapping $\G_{\b}:\H^{(\l+2\b)}\ra\text{Hol}(\D^n,\C^r)$ for $\b\in\L$ as follows 
$$\G_{\b}(f)=\left({\gamma\choose\b}\frac{\del^{\gamma-\b}f}{(\l+2\b)_{\gamma-\b}}\right)_{0\leq \gamma\leq \a}.$$ 
Here, ${\gamma\choose\b}={\gamma_1\choose\b_1}\cdots {\gamma_n\choose\b_n}$ is the binomial symbol in multi-index and $ (\l+2\b)_{\gamma-\b} $ is the Pochhammer symbol in multi-index which by definition is $(\l_1+2\b_1)_{\gamma_1-\b_1}\cdots (\l_n+2\b_n)_{\gamma_n-\b_n} $.
Note that if $\gamma_i<\b_i$ for some $i$, then ${\gamma\choose\b}=0$. It follows from the definition that each $\G_{\b}$ is injective.  We define the mapping $\G$ on $\bigoplus_{0\leq\b\leq\a}\H^{(\l+2\b)}$ taking values in $\text{Hol}(\D^n,\C^r)$ by 
\begin{equation}\label{equation of gamma}
(f_{\b})_{0\leq \b\leq\a}\mapsto \sum_{0\leq\b\leq\a}\mu_{\b}\G_{\b}(f_{\b})
\end{equation}
 where $\mu_{\b}$ with $0\leq\b\leq\a$ are all positive real numbers and $\mu_0=1$. Note from the definition that $\G$ is an injective linear transformation.

\begin{lem}\label{lemma injective}
The map $\G$ on $\bigoplus_{0\leq\b\leq\a}\H^{(\l+2\b)}$ taking values in $\text{Hol}(\D^n,\C^r)$ defined by the equation \eqref{equation of gamma} is injective.
\end{lem}

\begin{proof}
Since $\L=\{\b\in(\N\cup\{0\})^n:\b\leq \a\}$ is a total ordered set, we write $\L = \{\beta_1, \ldots, \beta_r\}$, where $\beta_i < \beta_j$ if $1 \leq i < j \leq r$. Let $f = (f_{\beta_i})_{1 \leq i \leq r} \in \bigoplus_{0\leq\b\leq\a}\H^{(\l+2\b)}$ be such that $\G f = 0$. 

By an induction argument, we now prove that $f_{\beta_i} = 0$ for every $1 \leq i \leq r$. Since $\beta_1 < \beta_i$ for every $i \geq 2$, it follows from the definition of $\G_{\beta_i}$ that $\left(\G_{\beta_i}f_{\beta_i}\right)_{\beta_1} = 0$ for every $i \geq 2$. This implies from the definition of $\G$ that $f_{\beta_1} = 0$. 

Now assume that there exists $1 \leq j < r$ such that $f_{\beta_k} = 0$ for every $1 \leq k \leq j$. A similar argument as given in the previous paragraph implies that $f_{\beta_{j + 1}} = 0$. This proves that $f_{\beta_i} = 0$ for every $1 \leq i \leq r$. Therefore the map $\G$ is injective.
\end{proof}

Let $\H^{(\l,\mu)}$ denote the image of $\G$. Being injective, $\G$ defines an inner product on $\H^{(\l,\mu)}$ as follows 
$$\left\langle\G((f_{\b})_{0\leq \b\leq\a}),\G((g_{\b})_{0\leq \b\leq\a})\right\rangle := \<(f_{\b})_{0\leq \b\leq\a},(g_{\b})_{0\leq \b\leq\a}\>$$ 
making $\G$ a unitary transformation onto $\H^{(\l,\mu)}$. We now prove that $\H^{(\l,\mu)}$ is a reproducing kernel Hilbert space by first showing that  the image Im$(\G_{\b})$ of $\G_{\b}$ is so for each $\b\in\L$. 

\begin{prop}
Let $\b\in\L$. Then Im$(\G_{\b})$ is a reproducing kernel Hilbert space with the reproducing kernel 
$$K_{\b}(z,w)=\left(\!\!\!\left(\frac{{s\choose \b}{t\choose \b}}{(\l+2\b)_{s-\b}(\l+2\b)_{t-\b}}\del^{s-\b}\bar{\del}^{t-\b}K^{(\l+2\b)}(z,w)\right)\!\!\!\right)_{0\leq s,t\leq\a},~z,w\in\D^n$$ where $K^{(\l+2\b)}(z,w)$ is as in \eqref{rep ker1}.
\end{prop}

\begin{proof}
Let $\{e_{\theta}\in\C^r:\theta\in\L\}$ be the standard ordered basis for $\C^r$. Note that for $f_{\b}\in\H^{(\l+2\b)}$,
\begin{eqnarray*}
\<\G_{\b}f_{\b},K_{\b}(.,w)e_{\theta}\> &=& \left<\G_{\b}f_{\b},\G_{\b}\frac{{\theta \choose \b}}{(\l+2\b)_{\theta-\b}}\dbar^{\theta-\b}K^{(\l+2\b)}(.,w)\right>\\
&=& \frac{{\theta\choose\b}}{(\l+2\b)_{\theta-\b}}\del^{\theta-\b}f_{\b}(w)\\
&=& \<\G_{\b}f_{\b}(w),e_{\theta}\>
\end{eqnarray*}
completing the proof.
\end{proof}

\begin{cor}
$\H^{(\l,\mu)}$ is a reproducing kernel Hilbert space with the reproducing kernel \beq\label{rep-ker}K^{(\l,\mu)}(z,w)=\sum_{0\leq\b\leq\a}\mu_{\b}^2K_{\b}(z,w),~z,w\in\D^n.\eeq
\end{cor}

At this point, we remark that the Hilbert spaces $\H^{(\l, \mu)}$ and $\bigoplus_{0\leq\b\leq\a}\H^{(\l+2\b)}$ differ by their Hilbert space structures since both of them represent the same set as demonstrated in the following lemma.

\begin{lem}\label{lemma equality}
The Hilbert spaces $\H^{(\l, \mu)}$ and $\bigoplus_{0\leq\b\leq\a}\H^{(\l+2\b)}$ are equal as a set.
\end{lem}

\begin{proof}
First, we claim that for each $f \in \H^{(\l + 2\beta)}$ and $0 \leq \gamma \leq \alpha$, $\left(\G_{\beta} f \right)_\g \in \H^{(\l + 2\g)}$. Writing $f = \otimes_{i = 1}^n f_i$ with $f_i \in \H^{(\l_i + 2\beta_i)}(\D)$, $ 1 \leq i \leq n$, we observe that $\partial_i^{\gamma_i - \beta_i}f_i \in \H^{(\l_i + 2\gamma_i)}(\D)$ whenever $\gamma_i - \beta_i \geq 0$ for every $ 1 \leq i \leq n$. Here, we use the fact that for $0 < \mu < \infty$ and $g \in \H^{(\mu)}(\D)$, $g' \in \H^{(\mu + 2)}(\D)$, where $\H^{(\mu)}(\D)$ is the reproducing kernel Hilbert space with the reproducing kernel $(1 - z\bar{w})^{ - \mu}$, $z, w \in \mathbb{D}$. Thus it yields that $\left(\G_{\beta}f\right)_{\gamma} \in \H^{(\l + 2\gamma)}$ verifying our claim since $\{f = \otimes_{i = 1}^n f_i : f_i \in \H^{(\l_i + 2\beta_i)}(\D), 1 \leq i \leq n\}$ is a total set in $\H^{(\l + 2\beta)}$.

Let $g \in \H^{(\l, \mu)}$ be an arbitrary element. Then there exists $f = (f_\beta)_{0 \leq \beta \leq \alpha} \in \bigoplus_{0\leq\b\leq\a}\H^{(\l+2\b)}$ such that $g = \G f$. From the definition of the $\G$ and the claim above, it follows that $\G f \in \bigoplus_{0\leq\b\leq\a}\H^{(\l+2\b)}$ verifying that $\H^{(\l, \mu)} \subseteq \bigoplus_{0\leq\b\leq\a}\H^{(\l+2\b)}$.

For the converse inclusion $\bigoplus_{0\leq\b\leq\a}\H^{(\l+2\b)} \subseteq \H^{(\l, \mu)}$, we first write $\L = \{\beta_1, \ldots, \beta_r\}$, where $\beta_i < \beta_j$ if $1 \leq i < j \leq r$ ($r$ is the cardinality of the set $\L$) since $\L=\{\b\in(\N\cup\{0\})^n:\b\leq \a\}$ is a total ordered set. From the definition of $\G$, note that $\{0\} \oplus \cdots \oplus \{0\} \oplus \H^{(\l + 2\beta_r)} \subseteq \H^{(\l, \mu)}$. Now we show that $\{0\} \oplus \cdots \oplus \{0\} \oplus \H^{(\l + \beta_{r-1})} \oplus \{0\}$ is contained in $\H^{(\l, \mu)}$. Let $f_{\beta_{r-1}} \in \H^{(\l + 2\beta_{r-1})}$ and consider the element $f = (f_{\beta_i})_{1 \leq i \leq r}$ where $f_{\beta_i} = 0$ for $ i \neq r-1$. Then we have that $\G f = \mu_{\beta_{r-1}} \G_{\beta_{r-1}} f_{\beta_{r-1}} \in \H^{(\l,\mu)}$. Furthermore, it follows from the claim above that $\left( \G_{\beta_{r-1}} f_{\beta_{r-1}}\right)_{\beta_r} \in \H^{(\l + 2\beta_r)}$. Consequently, we have that 
$$(0, \ldots , 0, f_{\beta_{r-1}}, 0) = \G f - \mu_{\beta_{r-1}}(0, \ldots, 0, \left( \G_{\beta_{r-1}} f_{\beta_{r-1}}\right)_{\beta_r}) \in \H^{(\l, \mu)} $$ 
since $\{0\} \oplus \cdots \oplus \{0\} \oplus \H^{(\l + 2\beta_r)} \subseteq \H^{(\l, \mu)}$ and $\left( \G_{\beta_{r-1}} f_{\beta_{r-1}}\right)_{\beta_{r-1}} = f_{\beta_{r-1}}$. This proves that $ \{0\} \oplus \cdots \oplus \{0\} \oplus \H^{(\l + 2\beta_{r-1})} \oplus \{0\} \subseteq \H^{(\l, \mu)}$. Finally, a similar argument with the help of mathematical induction yields that $\{0\} \oplus \cdots \oplus \{0\} \oplus \H^{(\l + 2\beta_j)} \oplus \{0\} \oplus \cdots \oplus \{0\} \subseteq \H^{(\l, \mu)}$ for every $1 \leq j \leq r$ which completes the proof.
\end{proof}

\begin{rem}
From the Lemma \ref{lemma equality} it follows that the multiplication operators by the coordinate functions on $\H^{(\l, \mu)}$ are well-defined (consequently, they are bounded). However, another proof of the boundedness of these operators are given in the Theorem \ref{boundedness}. 
\end{rem}

Next, we prove that the reproducing kernel $K^{(\l,\mu)}$ is quasi-invariant \w $\widetilde{\mob}^n$, that is, $K^{(\l,\mu)}$ satisfies the equation \eqref{quasi-invariant} for some co-cycle $J:\widetilde{\mob}^n\times\D^n\ra\text{GL}(r,\C)$. This can be achieved by producing an unitary multiplier representation of $\widetilde{\mob}^n$ onto $\H^{(\l,\mu)}$ defined by the equation \eqref{mult-rep} with $J$ as the multiplier. 

In the case of reproducing kernel Hilbert spaces of scalar valued functions on $\D$ (that is, for $n=r=1$), the unitary multiplier representations of $\mob$ are well known. Since these are the building blocks of the Hilbert spaces of our interest, we first describe them. These are the elements of holomorphic discrete series representation $D_{\l}^+$ of $\mob$ acting on the weighted Bergman space $\H^{(\l)}(\D)$ possessing a reproducing kernel $B^{(\l)}(z,w)=(1-z\ov{w})^{-\l}$ for $z,w\in \D$ and $\l>0$.  

For $g\in \mob$ and $t\in\R_{>0}$, taking the principal branch of power function we can uniquely define $g'(z)^{t}$ as a holomorphic function on $\D$. Now for $\tilde{g}\in\widetilde{\mob}$ with $p(\tilde{g})=g\in\mob$ ($p$ is the universal covering map), the multiplier $j_{t}(\tilde{g},z)=g'(z)^{t}$ defines on $\H^{(t)}(\D)$ the unitary representation $D^{(t)}_+$ as follows:
$$D^{(t)}_+(\tilde{g}^{-1})(f):=(g')^{\frac{t}{2}}(f\circ g), ~~~f\in\H^{(t)}(\D),~\tilde{g}\in\widetilde{\mob}.$$ As a consequence, for $f_{\b}\in\H^{(\l+2\b)}$ and $\tilde{\phi}\in\widetilde{\mob}^n$,  $$D_+^{(\l+2\b)}(\tilde{\phi}^{-1})(f_{\b})=(\phi')^{\frac{\l+2\b}{2}}(f_{\b}\circ\phi)=\prod_{i=1}^n(\phi_i')^{\frac{\l_i+2\b_i}{2}}(f_{\b_i}\circ\phi_i)$$ defines a unitary multiplier representation of $\widetilde{\mob}^n$ onto $\H^{(\l+2\b)}$ for each $\b\in\L$. Therefore, the direct sum of these representations can be transferred to $\H^{(\l,\mu)}$ by the map $\G$. We show that this is a multiplier representation. 

We need a relation between $g''(z)$ and $g'(z)$, for $g\in \mob$, in what follows. Note that for $g\in\mob$,
\beq\label{eqn9}
g''(z)=-2c_gg'(z)^{\frac{3}{2}}
\eeq 
where $c_g$ is independent of $z$ and the meaning of $g'(z)^{\frac{3}{2}}$ is as defined above. 

\begin{prop}\label{Theorem on mult-rep}
For $\l=(\l_1,\hdots,\l_n)$ with each $\l_i>0$, the image of $\bigoplus_{0\leq\b\leq\a}D_+^{(\l+2\b)}$ under the mapping $\G$ is a multiplier representation of $\widetilde{\mob}^n$ onto $\H^{(\l,\mu)}$ with the multiplier \beq\label{cocycle}
J(\tilde{\phi},z)_{\theta\eta} =
\left\{
	\begin{array}{lll}
		 {\theta\choose\eta}(-c_{\phi})^{\theta-\eta}(\phi')^{\frac{\l+\theta+\eta}{2}}(z) & \mbox{if}~~ 0\leq\eta_i\leq\theta_i,~i=1,\hdots,n \\
		0 & \mbox{otherwise}.
	\end{array}
\right.
\eeq
\end{prop}

\begin{proof}
We begin with the observation that it is enough to show, for any elementary tensors $f_{\b}=f_{\b_1}\otimes\cdots\otimes f_{\b_n}\in\H^{(\l+2\b)}$ with $f_{\b_i}\in\H^{(\l_i+2\b_i)}(\D)$, $i=1,\hdots,n$, the following identity: $$\G_{\b}(D_+^{(\l+2\b)}(\tilde{\phi}^{-1})f_{\b})=J(\tilde{\phi},.)\G_{\b}(f_{\b}),~\b \in \L,~\text{and}~\tilde{\phi}\in\widetilde{\mob}^n.$$ Now using the definition of $\G_{\b}$ note that 
\beq\label{eqn1}
\G_{\b}(D_+^{(\l+2\b)}(\tilde{\phi}^{-1})f_{\b})=\left({\g\choose\b}\frac{\del^{\g-\b}((\phi')^{\frac{\l+2\b}{2}}f_{\b}\circ\phi)}{(\l+2\b)_{\g-\b}}\right)_{0\leq \g\leq \a}\eeq
where $$\del^{\g-\b}((\phi')^{\frac{\l+2\b}{2}}f_{\b}\circ\phi)=\prod_{i=1}^n\del_i^{\g_i-\b_i}\left((\phi_i')^{\frac{\l_i+2\b_i}{2}}f_{\b_i}\circ\phi_i\right).$$ For $\g\in\L$ with $\g_i<\b_i$ for some $1\leq i\leq n$, since ${\g_i\choose\b_i}=0$ it is seen that $(\G_{\b}(f))_{\g}=0$.  Now using the formula (cf. \cite[Lemma 3.1]{HOHS})
$$\del^k((g')^lf\circ g)=\sum_{i=0}^k{k\choose i}(2l+i)_{k-i}(-c_g)^{k-i}(g')^{l+\frac{k+i}{2}}(f^{(i)}\circ g)$$ 
where $g''=-2c_g(g')^{\frac{3}{2}}$ for $g\in\mob$, we have that
\begin{align*}
&\prod_{i=1}^n\del_i^{\g_i-\b_i}\left((\phi_i')^{\frac{\l_i+2\b_i}{2}}f_{\b_i}\circ\phi_i\right)\\
&=\prod_{i=1}^n\sum_{t_i=0}^{\g_i-\b_i}{{\g_i-\b_i}\choose t_i}(\l_i+2\b_i+t_i)_{\g_i-\b_i-t_i}(-c_{\phi_i})^{\g_i-\b_i-t_i}\\
&\times\left((\del_i^{t_i}f_{\b_i})\circ\phi_i\right)(\phi'_i)^{\frac{\l_i+\g_i+\b_i+t_i}{2}}\\
&=\prod_{i=1}^n\sum_{t_i+\b_i=\b_i}^{\g_i}{{\g_i-\b_i}\choose {t_i+\b_i-\b_i}}(\l_i+\b_i+t_i+\b_i)_{\g_i-(t_i+\b_i)}(-c_{\phi_i})^{\g_i-(t_i+\b_i)}\\
&\times \left((\del_i^{t_i+\b_i-\b_i}f_{\b_i})\circ\phi_i\right)(\phi'_i)^{\frac{\l_i+\g_i+(t_i+\b_i)}{2}}\\
&=\prod_{i=1}^n\sum_{s_i=\b_i}^{\g_i}\!\!{{\g_i-\b_i}\choose {s_i-\b_i}}(\l_i+\b_i+s_i)_{\g_i-s_i}(-c_{\phi_i})^{\g_i-s_i}\left(\!\!(\del_i^{s_i-\b_i}f_{\b_i})\circ\phi_i\!\right)\!\!(\phi'_i)^{\frac{\l_i+\g_i+s_i}{2}}\\
&=\sum_{s_i=\b_i}^{\g_i}\prod_{i=1}^n{{\g_i-\b_i}\choose {s_i-\b_i}}(\l_i+\b_i+s_i)_{\g_i-s_i}(-c_{\phi_i})^{\g_i-s_i}\left((\del_i^{s_i-\b_i}f_{\b_i})\circ\phi_i\right)(\phi'_i)^{\frac{\l_i+\g_i+s_i}{2}}\\
&= \sum_{s=\b}^{\g}{{\g-\b}\choose {s-\b}}(\l+\b+s)_{\g-s}(-c_{\phi})^{\g-s}(\phi')^{\frac{\l+\g+s}{2}}((\del^{s-\b}f_{\b})\circ\phi).
\end{align*} 
Here $s$ denotes the tuple $(s_1,\hdots,s_n)$. Thus this computation together with \eqref{eqn1} yield that for $0\leq\b\leq\a$,
\begin{align*}
&(\G_{\b}(D_+^{(\l+2\b)}(\tilde{\phi}^{-1})f_{\b}))_{\g}\\
&=\frac{{\g\choose\b}}{(\l+2\b)_{\g-\b}}\sum_{s=\b}^{\g}{{\g-\b}\choose{s-\b}}(\l+\b+s)_{\g-s}(-c_{\phi})^{\g-s}(\phi')^{\frac{\l+\g+s}{2}}(\del^{s-\b}f_{\b})\circ\phi\\
&=\frac{1}{(\l+2\b)_{\g-\b}}\sum_{s=\b}^{\g}{\g\choose s}{s\choose\b}(\l+\b+s)_{\g-s}(-c_{\phi})^{\g-s}(\phi')^{\frac{\l+\g+s}{2}}(\del^{s-\b}f_{\b})\circ\phi
\end{align*}
where the last equality holds since ${\g\choose\b}{{\g-\b}\choose{s-\b}}={\g\choose s}{s\choose\b}$. Consequently, we have that 
\beq
\G_{\b}(D_+^{(\l+2\b)}(\tilde{\phi}^{-1})f_{\b})(z) =J(\tilde{\phi},z)\G_{\b}(f_{\b})
\eeq
where $J(\tilde{\phi},z)_{\theta\eta}$ is as given in \eqref{cocycle}. 
\end{proof}

\begin{thm}
Let $K^{(\l,\mu)}$ be the reproducing kernel of the Hilbert space $\H^{(\l,\mu)}$. Then 
\beq\label{quasi-inv}
K^{(\l,\mu)}(z,w) = J(\tilde{\phi},z)K^{(\l,\mu)}(\phi(z),\phi(w))J(\tilde{\phi},w)^*,~\tilde{\phi}\in\widetilde{\mob}^n,~z,w\in\D^n,
\eeq
where $J$ is the co-cycle defined by the equation \eqref{cocycle}.
\end{thm}

\begin{proof}
Let $\mathcal{U}(\H^{(\l,\mu)})$ be the set of unitary operators on $\H^{(\l,\mu)}$. We then define the mapping $U:\widetilde{\mob}^n\ra \mathcal{U}(\H^{(\l,\mu)})$ by the formula $$U(\tilde{\phi}^{-1})f:=J(\tilde{\phi},\cdot)f\circ \phi,~\tilde{\phi}\in\widetilde{\mob}^n,~f\in\H^{(\l,\mu)}.$$ Observe from Proposition \ref{Theorem on mult-rep} that $U$ is a unitary representation and consequently, the equation \eqref{quasi-inv} follows. \end{proof}

A canonical decomposition of the reproducing kernel $K^{(\l,\mu)}$ can be obtained using the identity \eqref{quasi-inv} and the expression of the co-cycle $J$ given in \eqref{cocycle}.  In order to obtain this decomposition, we need some notations. For $1\leq i\leq n$, denote $\tilde{S}_i$ to be the $i$-th shift defined as follows:
\beq\label{shift}\tilde{S}_ie_{\theta}=(\theta_i+1)e_{\theta+\varepsilon_i},~0\leq\theta\leq\a,\eeq 
and $D(z\ov{w})$ is the diagonal operator $D(z\ov{w})e_{\theta}=(1-z\ov{w})^{-\theta}e_{\theta},~0\leq \theta\leq\a$ where $\{\varepsilon_i:1\leq i\leq n\}$ and $\{e_{\theta}:\theta\in\L\}$ are the standard ordered bases of $\C^n$ and $\C^r$, respectively. For $z\in\D^n$, let $\tilde{\phi}_z\in\widetilde{\mob}^n$ with $p(\tilde{\phi}_z)=\phi_z\in\mob^n$ be the automorphism $\phi_z=(\phi_{z_1},\hdots,\phi_{z_n})$ where $p$ is the universal covering map and $\phi_{z_i}(w_i)=\frac{w_i-z_i}{1-\ov{z}_iw_i},~w_i\in\D,~1\leq i\leq n.$ Consequently, $c_{\phi_{z_i}}=-\frac{\ov{z}_i}{\sqrt{1-|z_i|^2}}$. The following lemma provides a canonical decomposition of the co-cycle $J$ introduced in \eqref{cocycle} in terms of the shift operators $\tilde{S}_i$, $1\leq i\leq n$ and the diagonal matrix $D(z\ov{w})$.




\begin{lem}\label{decomposition of cocycle}
The co-cycle $J$ defined in \eqref{cocycle} turns out to be 
\beq\label{factorisation of cocycle}
J(\phi_z,z)=(1-|z|^2)^{-\frac{\l}{2}}D(|z|^2)\exp\left(\sum_{i=1}^n\ov{z}_i\tilde{S}_i\right).
\eeq
\end{lem}

\begin{proof}
We begin by recalling from the equation \eqref{cocycle} that
$$J(\tilde{\phi},z)_{\theta\eta}={\theta\choose\eta}(-c_{\phi})^{\theta-\eta}(\phi')^{\frac{\l+\theta+\eta}{2}}(z)$$ whenever $0\leq\eta_i\leq\theta_i$. Substituting $\phi=\phi_z$ in this equation we have, for $0\leq\eta_i\leq\theta_i$ and $1\leq i\leq n$, that 
$$J(\tilde{\phi}_z,z)_{\theta\eta} ={\theta\choose\eta}\left(\frac{\ov{z}}{\sqrt{1-|z|^2}}\right)^{\theta-\eta}\left(\frac{1-|z|^2}{(1-|z|^2)^2}\right)^{\frac{\l+\theta+\eta}{2}}= (1-|z|^2)^{-\frac{\l}{2}}{\theta\choose\eta}\ov{z}^{\theta-\eta}(1-|z|^2)^{-\theta}.$$
On the other hand, note that
\Bea
D(|z|^2)\exp\left(\sum_{i=1}^n\ov{z}_i\tilde{S}_i\right)e_{\eta} &=& D(|z|^2)\exp(\ov{z}_n\tilde{S}_n)\cdots\exp(\ov{z}_1\tilde{S}_1)e_{\eta}\\
&= & D(|z|^2)\sum_{k=0}^{\a}\ov{z}^k{{\eta+k}\choose k}e_{\eta+k}\\
&=& D(|z|^2)\sum_{\theta=\eta}^{\a}{\theta\choose \eta}\ov{z}^{\theta-\eta}e_{\theta}\\
&=& \sum_{\theta=\eta}^{\a}{\theta\choose \eta}\ov{z}^{\theta-\eta}(1-|z|^2)^{-\theta}e_{\theta}
\Eea
where the second last equality is obtained by replacing $\eta+k$ to $\theta$ and using the identity ${\theta\choose \theta-\eta}={\theta\choose \eta}$. Thus it proves the desired identity.
\end{proof}

\begin{prop}\label{canonical decomposition of kernel}
The reproducing kernel $K^{(\l,\mu)}$ of the Hilbert space $\H^{(\l,\mu)}$ is of the form $$K^{(\l,\mu)}(z,w)=(1-z\ov{w})^{-\l}D(z\ov{w})\exp\left(\sum_{i=1}^n\ov{w}_i\tilde{S}_i\right)K^{(\l,\mu)}(0,0)\exp\left(\sum_{i=1}^nz_i\tilde{S}_i^*\right)D(z\ov{w})$$ for all $z,w\in\D^n$ where $K^{(\l,\mu)}(0,0)$ is a positive diagonal matrix.
\end{prop}

\begin{proof}
The proof follows from Lemma \ref{decomposition of cocycle} and Equation \eqref{quasi-inv}.
\end{proof}

We now prove the boundedness of the multiplication operators on $\H^{(\l,\mu)}$ using the following well-known lemma. Recall that the multiplication operator $M^{(\epsilon)}$ on $\H^{(\epsilon)}(\D)$ defined by $(M^{(\epsilon)}f)(z)=zf(z)$ is bounded for any $\epsilon>0$. Consequently, the reproducing kernel $K^{(\epsilon)}$ of the Hilbert space $\H^{(\epsilon)}(\D)$ satisfies the inequality \eqref{inq4} in  Lemma \ref{lem3} for any $\epsilon>0$. We use this technique to show that the $n$ - tuple of the multiplication operators by the co-ordinate functions on $\H^{(\l,\mu)}$ are bounded.

\begin{lem}\label{lem3}
Let $\H_K$ be a reproducing kernel Hilbert space with the reproducing kernel $K$ on some domain in $\C^n$. Then for $1\leq j\leq n$, the multiplication operator $M_{z_j}$ is bounded if and only if there exists a positive constant $c_j$ such that \beq\label{inq4}(c_j - z_j\ov{w}_j)K(\bz,\bw)\geq 0.\eeq
\end{lem}

\begin{thm}\label{boundedness}
The multiplication operators $M_{z_1},\hdots,M_{z_n}$ by the co-ordinate functions on $\H^{(\l,\mu)}$ are bounded.
\end{thm}

\begin{proof}
We begin by recalling that the multiplication operators $M_{z_1},\hdots,M_{z_n}$ by the co-ordinate functions on the reproducing kernel Hilbert space $\H^{(\epsilon)}(\D^n)$ with the reproducing kernel $K^{(\epsilon)}(z,w)$ $=\prod_{i}^n(1-z_i\ov{w}_i)^{-\epsilon_i}$ for $z,w\in\D^n$ are bounded for all $\epsilon=(\epsilon_1,\hdots,\epsilon_n)\in\mathbb{R}_{>0}^n$. Consequently, it follows from Lemma \ref{lem3} that for each $i=1,\hdots,n$, there exists a positive constant $c_i>0$ such that the inequality in \eqref{inq4} holds with $K=K^{(\epsilon)}$. In view of this observation together with Lemma \ref{lem3}, it is enough to write down the kernel $K^{(\l,\mu)}$ as a product of $K^{(\epsilon)}$ and a positive definite kernel $K_1$ on $\D^n$ for some $\epsilon_1,\hdots,\epsilon_n>0$.  Note that the expression of $K^{(\l,\mu)}$ obtained in Theorem \ref{canonical decomposition of kernel} suggests $K^{(\l-\epsilon,\mu')}$ to be a possible choice for $K_1$ provided $K^{(\l-\epsilon,\mu')}(0,0)$ is same as $K^{(\l,\mu)}(0,0)$ for some positive real numbers $\mu'_{\b}$ for $0\leq\b\leq\a$ with $\mu_0'=1$ and $\epsilon>0$. In rest of the proof, we show the existence of such positive real numbers $\mu'_{\b}$ for $0\leq\b\leq\a$ and $\epsilon>0$.

For $0\leq\theta\leq\a$, recall from the equation \eqref{rep-ker} that the $\theta$-th entry of the diagonal matrix $K^{(\l,\mu)}(0,0)$ is $$K^{(\l,\mu)}(0,0)_{\theta\theta}=\sum_{0\leq\b\leq\theta}\mu_{\theta}^2K_{\b}(0,0)_{\theta\theta}$$ with $\mu_0=1$ and $$K_{\b}(z,w)=\left(\!\!\!\left(\frac{{s\choose\b}{t\choose\b}}{(\l+2\b)_{s-\b}(\l+2\b)_{t-\b}}\del^{s-\b}\dbar^{t-\b}K^{(\l+2\b)}(z,w)\right)\!\!\!\right)_{0\leq s,t\leq \a}.$$ Thus it follows that for each $0\leq\theta\leq\a$, $$K_{\b}(0,0)_{\theta\theta}=\frac{{\theta\choose\b}^2(\theta-\b)!}{(\l+2\b)_{\theta-\b}}$$ and consequently, 
\beq \label{K(0,0)}
K^{(\l,\mu)}(0,0)_{\theta\theta}=\sum_{0\leq\b\leq\theta}\mu_{\b}^2\frac{{\theta\choose\b}^2(\theta-\b)!}{(\l+2\b)_{\theta-\b}}.\eeq

Let $[K^{(\l,\mu)}(0,0)]$ be the column vector whose $\theta$-th entry is $K^{(\l,\mu)}(0,0)_{\theta\theta}$ for $0\leq\theta\leq\a$. Then observe from the equation \eqref{K(0,0)} that 
\beq\label{K(00)}[K^{(\l,\mu)}(0,0)]=L(\l)(\mu^2)\eeq
 where $\mu^2$ is the column vector with $\mu_{\b}^2$ in it's $\b$-th position for $0\leq\b\leq\a$ and $L(\l)$ is the lower triangular matrix whose $\theta\b$-th entry is $\frac{{\theta\choose\b}^2(\theta-\b)!}{(\l+2\b)_{\theta-\b}}$ for $0\leq\b\leq\theta$. Also, note that $L(\l)$ is invertible and continuous in $\l$ verifying that $L(\l-\varepsilon)^{-1}L(\l)(\mu^2)$ is a vector with all of it's entries being positive, for sufficiently small $\epsilon=(\epsilon_1,\hdots,\epsilon_n)$ with each $\epsilon_i>0$. Thus it completes the proof with the choice $(\mu')^2=L(\l-\epsilon)^{-1}L(\l)(\mu^2)$.
\end{proof}

We conclude this section by showing that the adjoint of the $n$ - tuple of multiplication operators $(M_{z_1},\hdots,M_{z_n})$ by the co-ordinate functions on $\H^{(\l,\mu)}$ is in the Cowen-Douglas class $\mathrm B_r(\D^n)$. From now on, this $n$ - tuple $(M_{z_1},\hdots,M_{z_n})$ will be denoted by $\M^{(\l,\mu)}$.

\begin{thm}\label{CD class}
The adjoint of the $n$-tuple of multiplication operators $\M^{(\l,\mu)}$ is in $\mathrm B_r(\D^n)$ where $r$ is the cardinality of the set $\L=\{\b\in(\N\cup\{0\})^n:\b\leq \a\}$.
\end{thm}

\begin{proof}

From Lemma \ref{lemma equality} and a routine generalization of \cite[Theorem 5.1]{RKHS}, it follows that the identity map between $\H^{(\l, \mu)}$ and $\bigoplus_{0\leq\b\leq\a}\H^{(\l+2\b)}$ is invertible. Also, the identity map intertwines the multiplication operators on $\H^{(\l, \mu)}$ and $\bigoplus_{0\leq\b\leq\a}\H^{(\l+2\b)}$. Therefore, the adjoint of the tuple of multiplication operators ${\boldsymbol{M}^{(\l, \mu)}}^*$ on $\H^{(\l, \mu)}$ is similar to the adjoint of multiplication operators on $\bigoplus_{0\leq\b\leq\a}\H^{(\l+2\b)}$. Now the theorem follows from the fact that the Cowen-Douglas class $\mathrm B_r(\D^n)$ is invariant under the similarity.
\end{proof}

\section{Irreducibility}\label{irreducible}

In this section, we prove that the $n$ - tuples of multiplication operators $\M^{(\l,\mu)}$ are irreducible. The lemma below, modelled after Lemma 5.1 in \cite{HOHS}, can be proved exactly in the same way as in the original proof, so it is omitted.

\begin{lem}
Suppose that the $n$ - tuple of multiplication operators $\M=(M_{z_1},\hdots, M_{z_n})$ by the co-ordinate functions on a reproducing kernel Hilbert space $\H$ with the reproducing kernel $K$ is in $\mathrm B_r(\D^n)$. If there exists an orthogonal projection $X$ commuting with $\M$ then 
$$\Phi_X(z)K(z,w)~=~K(z,w)\ov{\Phi_X(w)}^{\text{tr}}$$ for some holomorphic function $\Phi_X:\D^n\ra\C^{r\times r}$ with $\Phi_X^2=\Phi_X$.
\end{lem}

Thus $\M^{(\l,\mu)}$ on $\H^{(\l,\mu)}$ is irreducible if and only if there is no non-trivial projection $X_0$ on $\C^r$ satisfying 
$$
X_0K^{(\l,\mu)}_0(z,0)^{-1}K^{(\l,\mu)}_0(z,w)K^{(\l,\mu)}_0(0,w)^{-1}~=~K^{(\l,\mu)}_0(z,0)^{-1}K^{(\l,\mu)}_0(z,w)K^{(\l,\mu)}_0(0,w)^{-1}X_0$$
where $K^{(\l,\mu)}_0(z,w)=K^{(\l,\mu)}(0,0)^{-\frac{1}{2}}K^{(\l,\mu)}(z,w)K^{(\l,\mu)}(0,0)^{-\frac{1}{2}}$. Let 
\beq\label{normalized kernel} \hat{K}^{(\l,\mu)}(z,w) = K^{(\l,\mu)}_0(z,0)^{-1}K^{(\l,\mu)}_0(z,w)K^{(\l,\mu)}_0(0,w)^{-1}\eeq which is called the normalized kernel associated to $K$ at origin. 

\begin{thm}\label{irreducibility}
The tuple of multiplication operators $\M^{(\l,\mu)}$ on $\H^{(\l,\mu)}$ are irreducible.
\end{thm}

\begin{proof}
In view of the discussion above, we show that there is no non-trivial projection commuting with the normalized kernel $\hat{K}^{(\l,\mu)}(z,w)$ for $z,w\in\D^n$ associated to the reproducing kernel $K^{(\l,\mu)}(z,w)$. Recall that $\L=\{\theta\in(\N\cup\{0\})^n:\theta\leq \a\}$ with $|\L|=r$ and $\{e_{\theta}:\theta\in \L\}$ is the standard ordered basis of $\C^r$. Denote the set $\{e_{\theta} : \theta \in \L\}$ by $\mathcal{E}$. For $1 \leq i \leq n$, let $\L_0^i=\{\theta\in\L:\theta_i=0\} = \left\lbrace \theta^{(0,i)}, \theta^{(1,i)}, \ldots, \theta^{(k,i)} \right\rbrace$ where $\theta^{(p,i)} \leq \theta^{(q,i)}$ if $p \leq q$. Then observe that the set $\mathcal{S}_{(\theta,i)}=\{\tilde{S}_i^je_{\theta}:j\geq 0\}$, $\theta\in\L_0^i$ is invariant under $\tilde{S}_i$. Since $1\leq i\le n$ is arbitrary but fixed for the rest of the proof, from now on we write $\mathcal{S}_{\theta}$ in place of $\mathcal{S}_{(\theta,i)}$. Note that $$\L=\bigcup_{\theta\in\L_0^i}\mathcal{S}_{\theta}.$$ 
Let $P$ be a permutation on $\mathcal{E}$ such that $P$ maps $\mathcal{E}$ to the ordered basis 
$$P\mathcal{E} = \left\lbrace S_i^j e_{\theta^{(0,i)}} : j \geq 0\right\rbrace \cup \cdots \cup \left\lbrace S_i^j e_{\theta^{(k,i)}} : j \geq 0 \right\rbrace$$
with $S_i e_\beta = e_{\beta + \epsilon_i}$ where we consider the ordering $S_i^le_{\theta^{(p,i)}}<S_i^je_{\theta^{(q,i)}}$ if $p<q$, or if $l<j$ whenever $p=q$.

For $z = (0, \ldots, z_i, \ldots, 0)$ and $w = (0, \ldots, w_i, \ldots, 0)$ in $\mathbb{D}^n$, let 
$$A(z_i, w_i) = \hat{K}^{(\l, \mu)}(z, w)$$
and observe that 
\begin{flalign*}
PA(z_i, w_i)P^{-1} = \displaystyle\bigoplus_{\theta \in \L_0^i} (1 - z_i\ov{w}_i)^{-m(\theta, i)} &\exp (-z_i \tilde{S}_{(\theta, i)}^\ast) B^{-1}_{(\theta, i)} \tilde{D}_\theta(z_i \bar{w}_i) \exp (\bar{w}_i \tilde{S}_{(\theta, i)})\times\\
&B_{(\theta, i)} \exp (z_i \tilde{S}_{(\theta, i)}^\ast) \tilde{D}_{\theta}(z_i\bar{w}_i) B_{(\theta, i)}^{-1} \exp (-\bar{w}_i \tilde{S}_{(\theta, i)}),
\end{flalign*}
where $B_{(\theta, i)} = K^{(\l, \mu)} (0,0)|_{\mathcal{S}_{\theta}}$, $m(\theta, i) = |\mathcal{S}_{\theta}| - 1$, $\tilde{S}_{(\theta,i)} = \tilde{S}_i|_{\mathcal{S}_{\theta}}$ and $\tilde{D}_\theta (z_i\bar{w}_i)$ is the diagonal operator defined as $\tilde{D}_\theta (z_i\bar{w}_i) S_i^j e_{\theta} = (1 - z_i\bar{w}_i)^{m(\theta, i) - j} e_{\theta + j\epsilon_i}$. We denote the expression 
$$\exp (-z_i \tilde{S}_{(\theta, i)}^\ast) B^{-1}_{(\theta, i)} \tilde{D}_\theta(z_i \bar{w}_i) \exp (\bar{w}_i \tilde{S}_{(\theta, i)}) B_{(\theta, i)} \exp (z_i \tilde{S}_{(\theta, i)}^\ast) \tilde{D}_{\theta}(z_i\bar{w}_i) B_{(\theta, i)}^{-1} \exp (-\bar{w}_i \tilde{S}_{(\theta, i)})$$
by $A_{\theta}(z_i, w_i)$ and note that
\begin{flalign*}
(1 - z_i\ov{w}_i)^{-m(\theta, i)}A_\theta(z_i, w_i) &= \left( \displaystyle \sum_{n\geq 0} a_{(\theta, i)}(n) (z_i\bar{w}_i)^n \right) \left( \displaystyle \sum_{p\geq 0} b_{(\theta, i)}(p) z_i^{p+1}\bar{w}_i^p + \cdots \right)\\
&= \displaystyle \sum_{q\geq 0} c_{(\theta, i)}(q) z_i^{q+1}\bar{w}_i^q + \cdots
\end{flalign*}
where $c_{(\theta, i)}(q) = \displaystyle \sum_{l = 0}^q a_{(\theta, i)}(l) b_{(\theta, i)}(q-l).$ It follows from \cite[Lemma 5.4]{HOHS} that each $b_{(\theta, i)}(p),\,\,p \geq 0,$ is a shift on $\C^{|\mathcal{S}_{\theta}|}$ which in turn implies that each $c_{(\theta. i)}(q),\,\,q \geq 0,$ is indeed a shift since each $a_{(\theta,i)}\in\C$.

For $0 \leq l \leq m(\theta, i)$, let $\mathcal{P}_l = \left\lbrace p \geq 0 : b_{(\theta, i)}(p)_{l, l+1} \neq 0 \right\rbrace$ which can be seen from \cite[Lemma 5.4]{HOHS} to be a non-empty set. Let $n_l = \mbox{min}\,\,\mathcal{P}_l$. We then have that
$$\left(c_{(\theta, i)}(n_l)\right)_{l,l+1} = \left( \displaystyle \sum_{n = 0}^{n_l} a_{(\theta, i)}(n) b_{(\theta, i)}(n_l - n)\right)_{l,l+1} \neq 0.$$
Thus for each $1 \leq i \leq n,$ $\theta \in \L_0^i$ and $0 \leq l \leq m(\theta, i)$, there exists $r$ co-efficient matrices $\displaystyle \oplus_{\theta \in \L_0^i} c_{(\theta, i)}(n_l)$ of $P A(z_i, w_i) P^{-1}$ whose $(l, l+1)$-th entry in the $\theta$'th block is non-zero.

Let $X$ be a projection commuting with $\hat{K}^{(\l,\mu)}$. Then note that $Y := PXP^{-1}$ is also a projection commuting with $P \hat{K}^{(\l,\mu)}P^{-1}$ since $P$ is a permutation matrix on the orthonormal basis $\mathcal{E}$ of $\mathbb{C}^r$. In particular, $Y$ commutes with $\displaystyle \oplus_{\theta \in \L_0^i} c_{(\theta, i)} (n_l)$ for all $1 \leq i \leq n$, $\theta \in \L_0^i$ and $0 \leq l \leq m(\theta, i)$. Thus $Y$ must be the identity operator verifying that so is $X$.
\end{proof}

\section{Inequivalence}

The objective of this section is to show that the operators in the family $\{\M^{(\l,\mu)}:\l\in\R^n_{>0},\mu\in\R^{r}_{>0},\mu_0=1\}$ are mutually inequivalent.  Recall from Theorem \ref{CD class} that for every $\l\in\R^n_{>0}$ and $\mu\in\R^{r}_{>0}$, the adjoint of the $n$ - tuple $\M^{(\l,\mu)}$ is in $\mathrm B_r(\D^n)$. Consequently, $\M^{(\l,\mu)}$ gives rise to a {\h} $E^{(\l,\mu)}$ over $\D^n$ with the hermitian structure induced from the reproducing kernel $K^{(\l,\mu)}$ as is well known (cf. \cite{OPOSE}).  Then for each $1\leq i,j \leq n$, the $ij$-th component of the curvature of the bundle $E^{(\l,\mu)}$ is defined by the formula $$\mathcal{K}^{(\l,\mu)}(z)_{ij}=\dbar_j(K^{(\l,\mu)}(z,z)^{-1}\del_i K^{(\l,\mu)}(z,z)),~z\in\D^n.$$

\begin{lem}\label{equiv1}
If $\M^{(\l,\mu)}$ and $\M^{(\l',\mu')}$ are unitarily equivalent then $\l=\l'$.
\end{lem}

\begin{proof}
Since $\M^{(\l,\mu)}$ and $\M^{(\l',\mu')}$ are unitarily equivalent, it follows (cf. \cite{OPOSE}) that $E^{(\l,\mu)}$ and $E^{(\l',\mu')}$ are isomorphic as {\h}s. Consequently, the curvature components of  $E^{(\l,\mu)}$ and $E^{(\l',\mu')}$ are similar. Note that the metric of the bundle $E^{(\l,\mu)}$ (resp. $E^{(\l',\mu')}$) is $K^{(\l,\mu)} (z, z)$ (resp. $K^{(\l',\mu')}(z,z)$), $z \in \D^n$. 

A direct computation shows that 
\begin{equation}\label{curvature}
\mathcal{K}^{(\l,\mu)}_{11}(0) = \bar{\partial}_1 \!\left(\! K^{(\l,\mu)} (z, z) \partial_1 K^{(\l,\mu)} (z,z)\! \right)\!\vert_{z=0} = \l_1 I + 2D_1 - \tilde{S}_1^* B^{-1}\tilde{S}_1B + B^{-1}\tilde{S}_1B\tilde{S}_1^*,
\end{equation}
where $B = K^{(\l,\mu)} (0,0)$ and $D_1 := \bar{\partial}_1 \partial_1 D(|z|^2)|_{z=0}$. Evaluating $\mathcal{K}_{11}^{(\l, \mu)}(0)$ on $e_\theta$, $0 \leq \theta \leq \alpha$, we obtain
$$\mathcal{K}_{11}^{(\l, \mu)}(0)e_\theta = \left(\l_1 + 2\theta_1 - \frac{(\theta_1 + 1)^2b_\theta}{b_{\theta + \epsilon_1}} + \frac{\theta_1^2 b_{\theta - \epsilon_1}}{b_\theta}\right)e_\theta,$$
where $b_\theta$ is the $\theta$'th diagonal entry of $B$. Thus we have
$$\mbox{Tr}\,\, \mathcal{K}_{11}^{(\l,\mu)}(0) = r\l_1 + 2\displaystyle \sum_{\theta \in \L} \theta_1.$$
Similarly, we have
$$\mbox{Tr}\,\, \mathcal{K}_{11}^{(\l',\mu')}(0) = r\l'_1 + 2\displaystyle \sum_{\theta \in \L} \theta_1.$$
Since $\mathcal{K}^{(\l,\mu)}_{11}(0)$ and $\mathcal{K}^{(\l',\mu')}_{11}(0)$ are similar, equating $\mbox{Tr}\,\, \mathcal{K}_{11}^{(\l,\mu)}(0)$ and $\mbox{Tr}\,\, \mathcal{K}_{11}^{(\l',\mu')}(0)$, we obtain $\l_1 = \l'_1$. A similar computation gives us $\l_i = \l'_i$ for all $2 \leq i \leq n$. 
\end{proof}

\begin{thm}
If $\M^{(\l,\mu)}$ and $\M^{(\l',\mu')}$ are unitarily equivalent then $\l=\l'$ and $\mu=\mu'$.
\end{thm}

\begin{proof}
In view of Lemma \ref{equiv1}, it is enough to show that $\mu=\mu'$ whenever $\M^{(\l,\mu)}$ and $\M^{(\l,\mu')}$ are unitarily equivalent.  Since by Theorem \ref{CD class} both $\M^{(\l,\mu)}$ and $\M^{(\l,\mu')}$ are in $\mathrm B_r(\D^n)$ and they are unitarily equivalent, it follows that the {\h}s $E^{(\l,\mu)}$ and $E^{(\l,\mu')}$ are isomorphic. Consequently, there exists a holomorphic function $\Phi$ on $\D^n$ taking values in $\text{GL}(r,\C)$ such that \beq\label{equiv of kernel} K^{(\l,\mu')}(z,w)=\Phi(z)K^{(\l,\mu)}(z,w)\Phi(w)^*,~\text{for all}~z,w\in\D^n.\eeq
Further, since both $\M^{(\l,\mu)}$ and $\M^{(\l,\mu')}$ are homogeneous \w $\mob^n$, and $\widetilde{\mob}^n$ action on $E^{(\l,\mu)}$ is unique (cf. \cite[Theorem 2.1]{ACHOCD}), it follows that the coycles $J$ and $J'$ associated to $E^{(\l,\mu)}$ and $E^{(\l,\mu')}$, respectively, satisfy
 \beq\label{equiv of cocycle} J(\tilde{g},z) = \Phi(g(z))^{-1}J'(\tilde{g},z)\Phi(z),~\tilde{g}\in\widetilde{\mob}^n,z\in\D^n.\eeq 
Indeed, note that the expression $ \Phi(g(\cdot))^{-1}J'(\tilde{g},\cdot)\Phi(\cdot)$ gives rise to an action of $\widetilde{\mob}^n$ on $E^{(\l,\mu)}$. In particular, for $g=(g_{z_1},\hdots,g_{z_n})\in\mob^n$ with $g_{z_i}(w_i)=\frac{w_i-z_i}{1-\ov{z}_iw_i}$, $i=1,\hdots,n$, we have from the Equation \eqref{factorisation of cocycle} that $$(1-|z|^2)^{-\frac{\l}{2}}D(|z|^2)\exp\left(\sum_{i=1}^n\ov{z}_i\tilde{S}_i\right)=(1-|z|^2)^{-\frac{\l}{2}}\Phi(0)^{-1}D(|z|^2)\exp\left(\sum_{i=1}^n\ov{z}_i\tilde{S}_i\right)\Phi(z),$$
 or equivalently, 
$$\Phi(z)=\exp\left(-\sum_{i=1}^n\ov{z}_i\tilde{S}_i\right)D(|z|^2)^{-1}\Phi(0)D(|z|^2)\exp\left(\sum_{i=1}^n\ov{z}_i\tilde{S}_i\right).$$ 
Note from this expression of $\Phi(z)$ that $\del^{p_1}_1\cdots\del^{p_n}_n\Phi(z)|_{z=0}=0$ for all $(p_1,\hdots,p_n)\in(\N\cup\{0\})^n$ with at least one $p_j\neq0$. Consequently, it follows that 
$$\Phi(z)=\Phi(0),~z\in\D^n$$
implying that 
\beq\label{eqn4.2.1}D(|z|^2)\exp\left(\sum_{i=1}^n\ov{z}_i\tilde{S}_i\right)\Phi(0)=\Phi(0)D(|z|^2)\exp\left(\sum_{i=1}^n\ov{z}_i\tilde{S}_i\right).\eeq
Now equating the coefficients of $|z_i|^2$ for each $i=1,\hdots,n$, in the Taylor expansion of both sides of this equation we have that 
$$\frac{\del^2}{\del z_i\dbar z_i}D(|z|^2)|_{z=0}\Phi(0)=\Phi(0)\frac{\del^2}{\del z_i\dbar z_i}D(|z|^2)|_{z=0},~i=1,\hdots,n.$$
Since $D(|z|^2)$ is the diagonal operator $D(|z|^2)e_{\theta}=(1-|z|^2)^{-\theta}e_{\theta},~0\leq \theta\leq\a$, the equation above implies that $\Phi(0)$ commutes with the diagonal matrices $D_1,\hdots,D_n$ where for $i=1,\hdots,n$, $D_i$ is the diagonal matrix whose $\b$-th diagonal entry is $\b_i$ for any $\b=(\b_1,\hdots,\b_n)\in\L$. Consequently, $\Phi(0)$ also commutes with all the linear combinations of the matrices $D_1,\hdots,D_n$. Note that there exist $c_1,\hdots,c_n\in\C$ such that the diagonal matrix $c_1D_1+\cdots+c_nD_n$ has distinct diagonal entries. Thus $\Phi(0)$ is a diagonal matrix. Moreover, by comparing the $\ov{z}_i$-th coefficients $\tilde{S}_i$ for each $i=1,\hdots,n$, in the Taylor expansion of both sides of the equation \eqref{eqn4.2.1}, it is seen that $\Phi(0)$ commutes with $\tilde{S}_i$, $i=1,\hdots,n$. Therefore, it follows from the definition of $\tilde{S}_i$, $i=1,\hdots,n$, there exists a non-zero $c\in\C$ such that $\Phi(0)=cI$ where $I$ is the $r\times r$ identity matrix.  

Finally, since $\Phi(0)$ satisfies the equation \eqref{equiv of kernel}, we have that $K^{\l,\mu}(0,0)=K^{(\l',\mu')}(0,0)$. Therefore, the equality $\mu=\mu'$ holds from the equation \eqref{K(0,0)} completing the proof.
\end{proof}

We recall the construction of a natural class of irreducible $n$ - tuple of operators in $\mathrm B_r(\D^n)$ that are homogeneous \w $\mob^n$ from the Introduction, namely, 
 \begin{equation} \label{tensor1}
\boldsymbol{T} = (\hat{T}^{(n_1)}, \ldots , \hat{T}^{(n_p)}),\,\,\mbox{\rm with}\,\, \hat{T}^{(n_i)} = I\otimes \cdots \otimes I \otimes \boldsymbol{T}^{(n_i)} \otimes I \otimes \cdots \otimes I,\,\, \boldsymbol{T}^{(n_i)}\in\mathrm B_{r_i} (\mathbb D^{n_i})
 \end{equation} 
with $n_1 + \cdots + n_p=n$, $r_1 \cdots r_{p} = r$ and $ \boldsymbol{T}^{( n_i )} = ( T_{1 n_i}, \hdots, T_{ n_i n_i } ) \in \mathrm B_{r_i} (\mathbb D^{n_i}) $. However, as promised in the Introduction, we now show that the $n$ - tuples $\boldsymbol{M}^{(\lambda, \mu)}$ are not of the form  \eqref{tensor1} if we insist that $n_i = 1,\, 1\leq i \leq n$. To establish this, we first prove a lemma. 

\begin{lem}\label{existence of theta}
Suppose $\alpha \neq \epsilon_1$. Then there exists $1 \leq i < j \leq n$ and an element $\theta$ in $\L$ such that $\theta - \epsilon_i + \epsilon_j$ is in $\L$, but $\theta + \epsilon_j$ is not in $\L$.
\end{lem}

\begin{proof}
We divide the proof into two cases. In the first case, we prove the existence of such $\theta, i, j$ by assuming $\alpha = (|\alpha|, 0, \ldots, 0)$. In the second case, we assume that $\alpha \neq (|\alpha|, 0, \ldots, 0)$ and prove that such $\theta, i, j$ exist.
\begin{itemize}
\item[Case 1:] Suppose $\alpha = (|\alpha|, 0, \ldots, 0)$. Since $\alpha \neq \epsilon_1$, it follows that $|\alpha| > 1$. Consider $\theta = \alpha - \epsilon_1$, $i = 1$ and any $j > 1$. Then we have $\theta - \epsilon_1 + \epsilon_j \in \L$, but $\theta + \epsilon_j \notin \L$.

\item[Case 2:] Assume that $\alpha \neq (|\alpha|, 0, \ldots, 0)$. There exists $k > 1$ such that $\alpha_k \geq 1$. Consider $\theta = (|\alpha|, 0, \ldots, 0)$, $i = 1$ and $j = k$. Since $|\theta + \epsilon_j| = |\alpha| + 1$, it follows that $\theta + \epsilon_j \notin \L$.

We claim that $\theta - \epsilon_1 + \epsilon_k \in \L$. On the contrary, assume $\theta - \epsilon_1 + \epsilon_k > \alpha$. Note that $|\theta - \epsilon_1 + \epsilon_k| = |\alpha|.$ Since $\alpha_k \geq 1$, we have $0 \leq (\theta - \epsilon_1 + \epsilon_k)_l \leq \alpha_l$ for every $l \geq 2$ and therefore, it follows from our assumption that $(\theta - \epsilon_1 + \epsilon_k)_l = \alpha_l$ for all $l \geq 2$. Also, the inequality $\theta - \epsilon_1 + \epsilon_k > \alpha$ implies that $\alpha_1 < (\theta - \epsilon_1 + \epsilon_k)_1 = |\alpha| - 1$. Thus 
$$|\alpha| = \displaystyle \sum_{l = 1}^n \alpha_l = \alpha_1 + \alpha_k < |\alpha| - 1 + 1 = |\alpha|,$$
which is a contradiction. This proves the claim.
\end{itemize}
The verification of these two cases completes the proof of the lemma.
\end{proof}

\begin{thm}\label{inequivalance with elimentary homog op}
For each $1 \leq i \leq n$, assume that adjoint of the multiplication operator on a reproducing kernel Hilbert space $\mathcal{H}_i$ with the reproducing kernel $K_i$ is a homogeneous operator in $\mathrm B_{r_i}(\D)$ and denote $M$ to be the $n$ - tuple of multiplication operators on $\displaystyle  \otimes_{i = 1}^n \mathcal{H}_i$. If the adjoint of $\M^{(\l,\mu)}$ is in $\mathrm B_r(\D^n)$ with $r > 2$ then $M$ and $\M^{(\l, \mu)}$ are unitarily inequivalent.
\end{thm}

\begin{proof}
Since $M^*$ is in $\mathrm B_{\prod_i r_i} (\D^n)$, it follows that if $\prod_i r_i \neq r$, then $M$ and $\M^{(\l,\mu)}$ are not unitarily equivalent. Therefore we assume that $\prod_i  r_i = r$.
Note that the adjoint of both the tuples $\M^{(\l,\mu)}$ and $M$ are in $\mathrm B_r(\D^n)$. Thus if $\M^{(\l, \mu)}$ and $M$ are unitarily equivalent, then the corresponding curvature components of $\hat{K}^{(\l,\mu)}$ and $\hat{K}$ are similar, where $\hat{K}^{(\l,\mu)}$ and $\hat{K}$ are normalized kernels of $K^{(\l,\mu)}$ and $K$ at origin, respectively (cf. \eqref{normalized kernel}).

A direct computation shows that 
\begin{equation}\label{curvature1}
\hat{\mathcal{K}}^{(\l,\mu)}_{ij}(0) = \bar{\partial}_j \partial_i \hat{K}^{(\l, \mu)}(z,z)|_{z=0} = B^{-\frac{1}{2}}\tilde{S}_jB\tilde{S}_i^*B^{-\frac{1}{2}} - B^{\frac{1}{2}}\tilde{S}_i^* B^{-1}\tilde{S}_j B^{\frac{1}{2}},
\end{equation}
for $1 \leq i \neq j \leq n$, where $B =K^{(\l,\mu)}(0,0)$ and $\tilde{S}_i$ is as in the equation \eqref{shift}. By Lemma \ref{existence of theta}, there exists a $\theta$ in $\L$ and $1 \leq p \neq q \leq n$ such that $\theta - \epsilon_p + \epsilon_q \in \L$, but $\theta + \epsilon_q \notin \L$. This implies that $B^{-\frac{1}{2}}\tilde{S}_qB\tilde{S}_p^*B^{-\frac{1}{2}}e_\theta \neq 0$, but $B^{\frac{1}{2}}\tilde{S}_p^* B^{-1}\tilde{S}_q B^{\frac{1}{2}} e_\theta = 0$. Therefore, we have $\hat{\mathcal{K}}^{(\l,\mu)}_{pq}(0) \neq 0$.

On the other hand, since $\hat{K} = \displaystyle \otimes_{i = 1}^n \hat{K}_i$, it follows that $\hat{\mathcal{K}}_{ij}(0) = \bar{\partial}_j \partial_i \hat{K}(z,z)|_{z=0} = 0$, for any $1 \leq i \neq j \leq n$. This implies that $\hat{\mathcal{K}}_{pq}(0)$ and $\hat{\mathcal{K}}^{(\l,\mu)}_{pq}(0)$ are not similar. Therefore, $M$ and $\M^{(\l,\mu)}$ are not unitarily equivalent.
\end{proof}

\medskip

\textit{Acknowledgement.} The authors are grateful to Professor Gadadhar Misra for his invaluable comments and suggestions in preparation of this article. The authors also thankful to the anonymous referee for the comments in improving the clarity of the present article.


\begin{thebibliography}{10}

\bibitem{HOPRMGS}
B.~Bagchi and G.~Misra, \emph{Homogeneous operators and projective
  representations of the {M}\"obius group: a survey}, Proc. Indian Acad. Sci.
  Math. Sci. \textbf{111} (2001), no.~4, 415--437.

\bibitem{THS}
\bysame, \emph{The homogeneous shifts}, J. Funct. Anal. \textbf{204} (2003),
  no.~2, 293--319.

\bibitem{OPOSE}
M.~J.~Cowen and R.~G.~Douglas, \emph{Operators possessing an open set of
  eigenvalues}, Colloq. Math. Soc. J\'anos Bolyai \textbf{I,II} (1980), no.~35,
  323--341.

\bibitem{GBKCD}
R.~E. ~Curto and N.~Salinas, \emph{Generalized {B}ergman kernels and the
  {C}owen-{D}ouglas theory}, Amer. J. Math. \textbf{106} (1984), no.~2,
  447--488.

\bibitem{HHHVB1}
P.~Deb and S.~Hazra, \emph{Homogeneous hermitian holomorphic vector bundles and
 operators in the cowen-douglas class over the poly-disc}, J. Math. Anal. Appl., 2022, vol: 510, no. 2, 32 pp. 


\bibitem{HOHS}
A.~Kor\'anyi and G.~Misra, \emph{Homogeneous operators on {H}ilbert spaces of
  holomorphic functions}, J. Funct. Anal. \textbf{254} (2008), no.~9,
  2419--2436.

\bibitem{ACHOCD}
\bysame, \emph{A classification of homogeneous operators in the
  {C}owen-{D}ouglas class}, Adv. Math. \textbf{226} (2011), no.~6, 5338--5360.

\bibitem{HHVCDBSYMD}
\bysame, \emph{Homogeneous {H}ermitian holomorphic vector
  bundles and the {C}owen-{D}ouglas class over bounded symmetric domains}, Adv. Math. \textbf{351} (2019), 1105--1138.

\bibitem{OICHO}
G.~Misra and S.~S. ~Roy, \emph{On the irreducibility of a class of homogeneous
  operators}, System theory, the {S}chur algorithm and multidimensional
  analysis, Oper. Theory Adv. Appl., vol. 176, 2007, pp.~165--198.

\bibitem{HVBIOSD}
G.~Misra and H.~Upmeier, \emph{Homogeneous vector bundles and intertwining
  operators for symmetric domains}, Adv. Math. \textbf{303} (2016), 1077--1121.

\bibitem{RKHS}
V. I. Paulsen and M. Raghupathi, \emph{An introduction to the theory of reproducing kernel {H}ilbert spaces}, Cambridge University Press, Cambridge, vol. 152, 2016.

\bibitem{HOJCS}
S.~S. ~Roy, \emph{Homogeneous operators, jet construction and similarity},
  Complex Anal. Oper. Theory \textbf{5} (2011), no.~1, 261--281.
  
  \bibitem{HVBCDO}
D.~R.~Wilkins, \emph{Homogeneous vector bundles and {C}owen-{D}ouglas
  operators}, Internat. J. Math. \textbf{4} (1993), no.~3, 503--520. 
 
\end{thebibliography}
\end{document}